\documentclass[11pt]{article}

\usepackage{amsfonts, amssymb}
\usepackage{amsmath}
\usepackage{color}
\usepackage{indentfirst}
\usepackage{verbatim}   
\usepackage{mathrsfs}   
\usepackage{yhmath}     

\begin{document}
\renewcommand{\theenumi}{(\roman{enumi})}
\newenvironment{proof}{\trivlist \item[\hskip \labelsep{\it Proof.}]}{
 \endtrivlist}
\newtheorem{corollary}{Corollary}[section]
\newtheorem{proposition}{Proposition}[section]
\newtheorem{theorem}{Theorem}[section]
\newtheorem{lemma}{Lemma}[section]
\newtheorem{remark}{Remark}[section]
\newtheorem{defi}{Definition}[section]
\newcounter{example}[section]
\newenvironment{example}{\refstepcounter{example} \trivlist
\item[\hskip \labelsep{\bf Example \theexample}]}{ \endtrivlist}
\newenvironment{claim}{\trivlist \item[\hskip \labelsep{\bf
Claim.} \it]}{ \endtrivlist}

\title  {Representations of the Renner Monoid
        }
\date{}
\maketitle

\vspace{ -1.6cm}
\begin{center}
\author Zhuo Li, \quad Zhenheng Li, \quad You'an Cao
\vskip 8mm
\end{center}

\def\J {{\cal J}}
\def\F {{\cal F}}
\def\a {\alpha}
\def\b {\beta}
\def\s {\sigma}
\def\v {\vec}
\def\js {(\cal J, \sigma)}
\def\ua {\underline a}
\def\ut {\underline t}
\def\e {\varepsilon}

\def\n { {\bf n} }
\def\r { {\bf r} }
\def\ek {\eta_{K}}
\def\eJ {\eta_{J}}
\newcommand{\be }{ \begin {}\\ \end{} }

\begin{abstract}
We describe irreducible representations and character formulas of the Renner monoids for reductive monoids, which generalizes the Munn-Solomon representation theory of rook monoids to any Renner monoids. The type map and polytope associated with reductive monoids play a crucial role in our work. It turns out that the irreducible representations of certain parabolic subgroups of the Weyl groups determine the complete set of irreducible representations of the Renner monoids. An analogue of the Munn-Solomon formula for calculating the character of the Renner monoids, in terms of the characters of the parabolic subgroups, is shown.

\vspace{ 0.3cm}

\vspace{ 0.3cm}

\end{abstract}

\baselineskip 18pt
\clubpenalty=1000
\widowpenalty=1000

\section{Introduction}
The Renner monoid generalizes the Weyl group from algebraic groups to monoids. A traditional example of the Renner monoid is the rook monoid $R_n$ consisting of 0's and 1's with each row and column having at most one 1. Another example is the so called symplectic Renner monoid introduced in \cite{LR}.

Munn \cite {M0, M1, M2} found the representation theory and character formula for $R_n$. Solomon \cite{LS2} reformulated the Munn theory for $R_n$ and determined the irreducible representations of $R_n$ in terms of the irreducible representations of certain symmetric groups by introducing some central idempotents in the monoid algebra $FR_n$ where $F$ is a field of characteristic 0. Using admissible sets and Solomon's approach, Li, Li, and Cao \cite{ZHEN} described the irreducible representations of symplectic Renner monoids by making use of the irreducible representations of symplectic Weyl groups and some symmetric groups.

Steinberg \cite{S1, S2} generalized Solomon's approach to finite inverse semigroups $S$ and proved that there is an algebra isomorphism between the groupiod algebra of $S$ and the monoid algebra of $S$. Using the M\"obius function on $S$ and adding up all the Solomon's idempotents in each $\mathscr{D}$-class of $S$, he found the decomposition of the monoid algebra into a direct sum of matrix algebras over group rings and obtained the character formula for multiplicities. His formula is versatile in that he gave applications to decomposing tensor powers and exterior products of rook matrix representations.

Putcha \cite{PU3} studied the representation theory of arbitrary finite monoids and determined all the irreducible characters of full transformation semigroups. He \cite{PU5} investigated highest weight categories and determined the blocks of the complex algebra of the full transformation semigroups. He found in \cite{PU2, PU4} the explicit isomorphism between the monoid algebra of the Renner monoid and the monoid Hecke algebra introduced by Solomon \cite{LS1}. Putcha and Renner \cite {PR2} studied the complex representation theory of a class of finite monoids $M$ and found its relationship with Harish-Chandra's theory of cuspidal representations of the unit group $G$ of $M$. Renner \cite {R2} showed that if $M$ is a finite monoid of Lie type with zero then the restriction of any irreducible representation of $M$ to $G$ is still irreducible. Furthermore, in \cite{R3} he computed the number of irreducible modular representations of $M$.

In this paper we characterize the irreducible representations and characters of the Renner monoid $R$ for reductive monoids $M$ with 0. Associated with $M$ are the type map and a polytope whose face lattice is isomorphic to the lattice of idempotents of $R$ \cite{P, P1, R4}. They are used in Theorem 3.1 to determine the decomposition of the monoid algebra of $R$ into a direct sum of matrix algebras. More specifically, let $\Lambda$ be a cross section lattice of $M$ and $W^*(e)$ and $W(e)$ parabolic subgroups of $W$ (for definitions, see Section 2.1 below). For any $e\in \Lambda$, denote by $B_e$ the group algebra of $W^*(e)$ over $F$. Let $d_e = |W|/|W(e)|$ and $M_{d_e}(B_e)$ the matrix algebra on $B_e$. Then the monoid algebra is a direct sum of all $M_{d_e}(B_e)$, where $e$ runs through $\Lambda$. Certain central idempotents in the monoid algebra, constructed from the polytope, are key. These results are used to determine all the inequivalent irreducible representations of $R$ in Theorem 3.2 in terms of the irreducible representations of $W^*(e)$ with $e\in \Lambda$.

We then in Theorem 4.1 give an analogue of the Munn-Solomon formula for calculating the character of the Renner monoid in terms of the characters of the parabolic groups $W^*(e)$ where $e \in \Lambda$. A couple of examples to show concrete irreducible representations are given. An example to demonstrate the character formula is shown at the end of Section 4. For more examples, we refer the reader to Solomon \cite{LS2} on rook monoids and Li, Li, Cao \cite{ZHEN} on symplectic rook monoids.

After the work of this article was outlined, the second named author met Dr. Steinberg in the Winter Meeting of the Canadian Mathematical Society in December, 2007, and learned that he independently generalized Solomon \cite{LS2}. For a finite inverse semigroup, he found the decomposition of the monoid algebra into a direct sum of matrix algebras over group rings via adding up all the Solomon idempotents in each $\mathcal D$-class. Our approach, however, makes use of cross section lattices, type maps, Weyl groups, and face lattices of polytopes. This is different from that of Steinberg. The main novelty in our work is an explicit computation of the M\"obius function of any Renner monoid in terms of dimensions of the faces of a polytope, a precise formula for Solomon central idempotents, and an ideal choice of the representatives of certain parabolic Weyl subgroups $W^*(e)$ with $e\in \Lambda$. This paper shows its close connection with algebraic groups, Lie theory (Weyl groups), and geometry (polytopes), to make it more accessible to the reader interested in algebraic groups and Lie theory.

\section{Preliminaries}
A {\it linear algebraic monoid} is an affine variety defined over an algebraically closed field $F$ together with an associative morphism and an identity. The unit group of an algebraic monoid is an algebraic group. By an irreducible algebraic monoid, we mean a linear algebraic monoid and irreducible as a variety. An irreducible monoid is called {\it reductive} if its unit group is a reductive algebraic group. The theory of reductive monoids has been developed by Putcha and Renner as a generalization of the theory of linear algebraic groups \cite{P1,  R4}. An example of reductive monoids is $M_n$, the multiplicative monoid of $n\times n$ matrices over $F$. Another example is the symplectic monoid studied in \cite{LR}.

\subsection{Type Map of Reductive Monoids}

Let $M$ be a reductive monoid with 0 and the group of units $G$. Let $B\subset G$ be a Borel subgroup, $T\subset B$ a maximal torus, and $W = N_G(T)/T$ the Weyl group. Let $\overline {N_G(T)} $ be the Zariski closure of $N_G(T)$ in $M$. Then $R= \overline {N_G(T)}/T$ is a finite inverse monoid with unit group $W$, and is called the Renner monoid. Let $\overline T$ be the Zariski closure of $T$ in $M$. Then
$$
    E(\overline T)=\{e\in \overline T \mid e^2 = e \}
$$
is a finite lattice of idempotents in $R$. There is a partial order on $E(\overline T)$; for $e, f\in
E(\overline T)$,
$
    f\le e \Longleftrightarrow ef=f=fe.
$
If $f\le e$ and $f \ne e$ then $f < e$. The set
$$
    \Lambda=\{e\in E(\overline T) \mid Be = eBe \}
$$
by definition is the {\it cross-section lattice} of $M$. Then $R = \langle W, ~\Lambda \rangle =W \Lambda W$. Suppose that $\Delta = \{\alpha_1, ..., \alpha_n\}$ is the set of simple roots relative to $T$ and $B$ and $S=\{s_{\a}\mid \a\in \Delta\}$ the set of simple reflections that generate $W$.

\begin{defi} The type map
$
    \lambda: \Lambda\to 2^{\Delta}
$
is defined by
\[
    \lambda(e) = \{\a\in\Delta \mid  s_{\a}e = es_{\a}\}.
\]
\end{defi}

To agree with Putcha and Renner \cite{P2, P3, R4}, let
$$
    \lambda^*(e) = \bigcap_{f\ge e}\lambda(f)  \qquad \lambda_*(e) = \bigcap_{f\le e}\lambda(f).
$$
Then
$$
\begin{aligned}
    &W(e)   = W_{\lambda(e)}   = \{w\in W \mid we=ew \} \\
    &W^*(e) = W_{\lambda^*(e)}, \\
    &W_*(e) = W_{\lambda_*(e)} = \{w\in W \mid we=ew = e \},
\end{aligned}
$$
are parabolic subgroups of $W$. The following proposition is from \cite {P1, PR1}.

\begin{proposition}\label{kepprop} Let $M$ be a reductive monoids with 0. For $e\in \Lambda$,
\[
    \begin{aligned}
        (i)~ &\lambda^*(e)=\{a\in \Delta\mid s_{\a}e=es_{\a}\not=e\}, \mbox{ and } \lambda_*(e)=\{a\in \Delta\mid s_{\a} e = e s_{\a} = e\}.   \\
        (ii)~ &\lambda(e)=\lambda^*(e)\sqcup \lambda_*(e).  \\
        (iii)~ &W(e) = W^*(e)\times W_*(e).  \\
        (iv)~ &W^*(e) = eW(e).
    \end{aligned}
\]
\end{proposition}

\begin{lemma} \label{lemma de} Use the notation above. Let $d_e = |W| / |W(e)| $. Then
$
    |WeW| = d_e^2 |W^*(e)|.
$
\end{lemma}

\begin{proof} By the remark that follows Theorem 2.1 of \cite {LLC} we obtain
\[
    |WeW| = \frac {|W|^2} { |W(e)|\times |W_*(e)| }.
\]
In view of $(iii)$ of Proposition \ref{kepprop} we have
$$
    |WeW| = \frac {|W|^2} { |W(e)|^2} \times |W^*(e)| = d_e^2 \times |W^*(e)|. \qquad \qquad\qquad\Box
$$
\end{proof}

A reductive monoid is $\J$-{\it irreducible} if its cross-section lattice has a unique minimal non-zero idempotent. The cross-section lattices of $\J$-irreducible monoids are completely determined by Putcha and Renner \cite {PR1}. It is summarized below.

\begin{theorem}\label{recipe}
(Theorem 4.16 of \cite{PR1}) Let $M$ be a $\J$-irreducible monoid associated with a dominant weight $\mu$ and $J=\{\a\in \Delta \mid s_{\a}\mu=\mu\}$. Then for $e\in \Lambda\setminus\{0\}$,
\begin{enumerate}
    \item $\lambda^*(e)=X\subset\Delta, \mbox{ where X has no connected component lies entirely in } J$.
    \item $\lambda_*(e)=\{\a\in J\setminus \lambda^*(e) \mid s_{\a}s_{\beta}=s_{\beta}s_{\a} \mbox{ for all } \beta \in \lambda^*(e)\}.$
    \item $\lambda^*(e_0) = \emptyset$ and $\lambda(e_0) = \lambda_*(e_0) = J$, where $e_0$ is the unique minimal element in $\Lambda$.
\end{enumerate}

\end{theorem}

\subsection{Face lattice of reductive monoids}
From now on, we always assume that $M$ is a reductive monoid with 0. According to Theorem 8.7 in \cite{P1} there is a polytope $P$ associated with $M$. Denote by $\F(P)$ the face lattice of $P$. We call $P$ the polytope of $M$ and $\F(P)$ the face lattice of $M$. It is a fact that $E(R) = E(\overline T) \simeq \F(P)$.  Denote by $V(P) = \{1, ..., m\}$ the set of the vertices of $P$. A subset $K \subseteq V(P)$ is referred to as a face of $P$ if $conv(K)$, the convex hull of $K$, is a face of $P$. By $\mbox{dim } K$ we mean $\mbox{dim\,} (conv(K))$. If $K$ is a face of $P$, denote by $\F(K)$ the set of faces $J$ of $P$ such that $J\subseteq K$. See \cite{BG, GZ} for more information on polytopes.

\begin{defi} Let $P$ be the polytope of $M$ and  $K$ a subset of $V(P)$. Define
\begin{eqnarray*}
    e_K (i) = \begin{cases}
                        i                       \qquad\qquad\qquad\mbox{ if } i\in K,\\
                        \text{undefined}        \quad\quad\mbox{ ~if } i\notin K.
                        \end{cases}
\end{eqnarray*}
By convention, $e_\emptyset = 0$, and $e_{V(P)}$ denotes the identity map.
\end{defi}

It follows from \cite{P, P1} that $\F(P)$ is isomorphic to $E(\overline T) = \{e_K \mid K\in \F(P)\}$ as lattices with the following isomorphism
\begin{eqnarray}\label{emap}
    \varepsilon:\quad K \mapsto e_K.
\end{eqnarray}
The set of vertices of the polytope corresponds to the set of minimal idempotents of $E(\overline T) - \{0\}$. Let W act by conjugation on $E(\overline T)$. Then we have
\begin{eqnarray} \label{eTbar}
    E(\overline T) = \bigsqcup_{e\in \Lambda} Cl_W(e),
\end{eqnarray}
where $Cl_W(e) = \{wew^{-1} \mid w\in W\}$, the $W$-conjugacy class of $e$. This action induces an action of $W$ on $\F(P)$ as follows. For $w\in W$ and $J, K \in \F(P)$,
\begin{eqnarray}\label{action}
    wJ = K, \mbox{ if  } w^{-1}e_Jw  = e_K.
\end{eqnarray}

\section{Representations}
For simplicity we assume in the remainder of this paper that the characteristic of $F$ is 0, even though most of the results are valid for positive characteristic. Let $R$ be the Renner monoid of $M$, and $A = FR$ the monoid algebra of $R$ over $F$. Clearly,
$$
    A = \Bigg\{\sum_{\sigma\in R} \alpha_\sigma \sigma  ~\Big |~ \alpha_\sigma \in F \Bigg\}.
$$
The zero element of $A$ is the linear combination with $\alpha_\sigma = 0$ for all $\sigma\in R$, and the identity element of $A$ is the element with $\alpha_\sigma = 0$ for all $\sigma\in R -  \{1\}$ but $\alpha_1 = 1 \in F$.

To investigate the representation theory of $A$, we define certain central idempotents of $A$ which were introduced by Solomon \cite {LS0} in his studying M\"{o}bius algebra of a lattice (see also \cite{ZHEN, LS2}).

Write $E=E(\overline T)$, and
let
\[
    FE =\oplus_{J \in \F(P)} Fe_J
\]
be the monoid algebra of $E$ over $F$. Thus, $FE \subseteq A$. For every $K\in \F(P)$, define
\begin{eqnarray}
    \eta_K =\sum_{J\in \F(K)}(-1)^{ \mbox{\footnotesize { dim\,}} K - \mbox{\footnotesize{ dim\,}} J }e_J \in FE.
\end{eqnarray}
Note that $\mu_{\F(P)}(J, K) = (-1)^{\dim K-\dim J}$, for $J\in \F(K)$, is the M\"{o}bius function of $\F(P)$ (see \cite {PU7}). It follows from the M\"{o}bius inversion \cite{RS} that
\begin{eqnarray}
    e_K =\sum_{J\in \F(K)}\eta_J.
\end{eqnarray}

\begin{lemma}\label{lemma JK}
    If $J$ and $K$ are faces of $P$, then $\eta_{K}\eta_{J} = \delta_{{J, K}} \eta_{J}$. So the $\eta_{K}$ are
    pairwise orthogonal idempotents of $A$.
\end{lemma}

\begin{proof}
We show first that, for any $J, K \in \F(P)$,
\begin{eqnarray*}
    e_K\eta_J = \bigg\{  \begin{aligned}     &\eta_J  \mbox{      \qquad if } J \subseteq K, \\
                                             &0       \mbox{      \qquad ~\,  otherwise }. \\
                        \end{aligned}
\end{eqnarray*}
If $J \subseteq K$, then $e_K e_I = e_{K\cap I} = e_I$ for $I \in \F(J)$, and
$$
    e_K\eta_J   = e_K \sum_{I \in \F(J)} (-1)^{\dim J - \dim I} e_I = \eta_J.
$$
If $J \nsubseteq K$, then for every $I\in \F(J)$, the face $X = I \cap K$ is a proper face of $J$. Let $M = \{I \cap K  \mid I \in \F(J) \}$. Then
\[
\begin{aligned}
     e_K\eta_J   &= \sum_{I \in \F(J)} (-1)^{\dim J - \dim I} e_{I\cap K}                      \\
                 &= \sum_{ X \in M } \bigg( \sum_{ {I \in \F(J)} \atop {I \cap K = X} } (-1)^{\dim J - \dim I}  \bigg) e_{X}   \\
                 &=0,
\end{aligned}
\]
since the inner sum is zero.
Therefore,
\[
    \eta_K\eta_J = \sum_{L \in \F(K)}(-1)^{\dim K - \dim L} e_L \eta_J = \sum_{J \subseteq L \in \F(K)}(-1)^{\dim K - \dim L} \eta_J.
\]
If $J = K$ then $\eta_K\eta_J = \eta_J$. If $J \ne K$ then $\eta_K\eta_J = 0$ by a generalized Euler's relation in Section 8.3 of \cite {BG}.      $\hfill\Box$
\end{proof}


\begin{lemma} \label {idempotent}

(i) For every $e \in E(\overline T)$ and $w\in W$, if $ew\in E(\overline T)$, then $ew = e$.

(ii) If $\sigma \in R$, there exist $w, w_1 \in W$ and unique $e, f\in E(\overline T)$ such that $\sigma = ew = w_1f$.
\end{lemma}

\begin{proof}
To show that $(i)$ is true, let $e_1 = ew$. Then $ee_1 = e_1$. It follows that $e_1\le e$. On the other hand, by $e_1 w^{-1} = e$ we have $e_1e = e$, which means that $e\le e_1$. Thus $e_1=e$. This proves (i). For $(ii)$, notice that $R$ is unit regular, i.e., $R = E(\overline T)W = W E(\overline T)$. There exist $e\in E(\overline T)$ and $w \in W$ such that
$
    \sigma = e w.
$
If $ew = e'w'$ for $e'\in E(\overline T)$ and $w' \in W$, then $e = e'w'w^{-1}$. By $(i)$, we have $e=e'$ and so $e$ is unique. Similarly, there exist $w_1 \in W$ and unique $f\in E(\overline T)$ such that $\sigma = w_1f$.
$\hfill\Box$
\end{proof}

It follows from (\ref{emap}) and Lemma \ref{idempotent} that, for any $\sigma \in R$, there exist $w, w_1 \in W$ and unique faces $I, J$ of $P$ such that $\sigma = e_Iw = w_1e_J$. This leads to the following definition.

\begin{defi} \label {domain}
Let $\sigma = e_I w = w_1 e_J \in R$ be as above. Then $I$ is called the domain of $\sigma$ and denoted by $I(\sigma)$; and $J$ the range of $\sigma$ and denoted by $J(\sigma)$.
\end{defi}

If $I$ and $J$ are the domain and range of $\sigma$ respectively, then $\sigma = e_I w e_J$ and maps $I$ to $J$. The product of $\sigma, \tau \in R$ is regarded as $i(\sigma\tau) = (i\sigma)\tau$ if $i\in I(\sigma)$ and $i\sigma\in I(\tau)$, for $i\in V(P)$. If $w\in W$ then $I(w) = J(w) = V(P)$.

\begin{lemma} \label{conjugate}

(i) $E(WeW) = Cl_W(e)$.

(ii) If $\sigma \in WeW$ and $I=I(\sigma), J=J(\sigma)$, then $e_I, e_J \in Cl_W(e)$.

(iii) If $f\in Cl_W(e)$ and $f\ne e$, then $ef < e$ and $ef < f$. In other words, different elements in $Cl_W(e)$ are not comparable.
\end{lemma}

\begin{proof}
Clearly, $Cl_W(e)  = \{wew^{-1} \mid w\in W\} \subseteq E(WeW)$. Let $e^w = wew^{-1}$. Then for any idempotent $e_K \in E(WeW)$, there exist $w_1, w_2 \in W$ such that $e_K = w_1 e w_2 = e^{w_1} w_1w_2$. Then $e_K = e^{w_1} \in Cl_W(e)$ by $(i)$ of Lemma \ref{idempotent}. Thus $E(WeW) \subseteq Cl_W(e)$. This proves $(i)$.

To show $(ii)$, for any $\sigma \in WeW$, we have $\sigma = w_1 e w_2$ for some $w_1, w_2 \in W$.  Let $I$ be the domain of $\sigma$. Then there exists $w\in W$ such that $\sigma = e_I w$. Thus, $e_I = w_1 e w_2 w^{-1}\in Cl_W(e)$ by $(i)$. Similarly, $e_J \in Cl_W(e)$. So $(ii)$ is true.

For $(iii)$ it suffices to show that $ef \ne f$ and $ef \ne e$. If $ef = f$, then $f < e$ since $f\ne e$. Then $f$ is conjugate to some $f'\in \Lambda$ with $f' < e$, which contradicts $f\in Cl_W(e)$. Therefore, $ef \ne f$. Similarly, $ef \ne e$. This proves $(iii)$.           $\hfill\Box$
\end{proof}

Denote by $L = \varepsilon^{-1}(e)$ the unique face of $P$ with $e = e_L$, where $\varepsilon$ is as in (\ref{emap}). Let \begin{eqnarray}
    \F(e) = \{wL \mid w\in W\}
\end{eqnarray}
be the orbit of $L$ under the action of $W$ on $P$ given in (\ref{action}). From $(i)$ of Lemma \ref {conjugate}, we see $\F(e)$ consists of faces $K$ of $P$ with $e_K \in Cl_W(e)$. In fact, $\varepsilon(\F(e)) = Cl_W(e)$.

Let
\begin{eqnarray}
    \eta_e = \sum_{K \in \F(e)}\eta_K.
\end{eqnarray}

\begin{lemma} \label{centralizer} For $e, f \in \Lambda$, we have

(i) $\eta_e \eta_f = \delta_{e f} \eta_f$.

(ii) The $\eta_e$ centralizes the Renner monoid $R$.
\end{lemma}

\begin{proof}
If $e \ne f$, then $WeW \cap WfW = \emptyset$. It follows from Lemma \ref{lemma JK} that $\eta_e \eta_f = 0$. If $e = f$, it is clear $\eta_e \eta_f = \eta_f$. Thus $(i)$ is correct. To show (ii), notice that for any $w\in W$, $w\eta_e w^{-1} = \sum_{K\in \F(e)} w\eta_K w^{-1} = \sum_{K\in \F(e)} \eta_{wK} = \eta_e$, i.e., $\eta_e$ centralizes $W$. On the other hand, it is fairly easy to check $\eta_e$ centralizes $E=E(\overline{T})$. But then $R = WE$ forces $\eta_e$ centralizes $R$.               $\hfill\Box$
\end{proof}

Lemma \ref{centralizer} indicates that $\eta_e$ is in the center of $A$ and so $A\eta_e$ is a two-sided ideal of $A$. Since $e_P = \sum_{K\in \F(P)} \eta_K = 1_A \in A$ is the identity element of $A$ and $\F(P) = \bigsqcup_{e\in \Lambda}\F(e)$, we see $1_A = \bigoplus_{e\in \Lambda} \eta_e$ and
\begin{eqnarray}\label{Adecompositin}
    A = \bigoplus_{e\in \Lambda} A\eta_e
\end{eqnarray}
is a direct sum of two-sided ideals. Let
\begin{eqnarray}
        \Lambda_*(e) &=& \{ f\in \Lambda \mid f\le e\}      \\
        R_*(e)       &=& \bigcup_{f\in \Lambda_*(e)} WfW    \\
        I_*(e)       &=& \sum_{\sigma \in R_*(e) } F\sigma \label{ie}.
\end{eqnarray} \label{istare}
The following Lemma \ref{lemma ideals} is elementary; we omit its simple proof.
\begin{lemma} \label{lemma ideals}
    (i) $\Lambda_*(e)$ is an ideal of $\Lambda$ as monoids.

    (ii) $R_*(e)$ is an ideal of $R$ as monoids.

    (iii) $I_*(e)$ is an ideal of $A$ as algebras.
\end{lemma}

\begin{lemma} \label{basis} Use the notation above.

(i) Let $e, f\in \Lambda$ and $\sigma\in WfW$. If $f<e$ or $f$ is not related to $e$, then $\sigma\eta_e = 0$.

(ii) $I_*(e) = \bigoplus_{f \le e} A\eta_f$.

(iii) The set $\{ \sigma \eta_e \mid \sigma\in WeW \}$ is an $F$-basis of $A\eta_e$.
\end{lemma}
\begin{proof}
Let $K$ be the face of $P$ such that $f=e_K$ as in (\ref{emap}). If $f<e$ or $f$ is not related to $e$, then for all $J\in \F(e)$ we have $J \nsubseteq K$ (otherwise, $e\le f$). It follows from (\ref{eKetaJ}) that
$$
    f\eta_e = e_K\eta_e = e_K \sum_{J \in \F(e)} \eta_J = \sum_{J \in \F(e)} e_K \eta_J = 0.
$$
The result $(i)$ follows from $(ii)$ of Lemma \ref{centralizer}.

To show $(ii)$, let $A_*{(e)} = \sum_{f \le e} A\eta_f$. If $f\le e$, then $\eta_f \in I_*{(f)} \subseteq I_*(e)$. Then $A_*{(e)} \subseteq I_*(e)$. Now we prove the inverse inclusion. It suffices to show that if $\sigma \in WfW$ with $f\le e$, then $\sigma \in A_*{(e)}$. There are $w_1, w_2 \in W$ such that $\sigma = w_1 f w_2$, since $R = W\Lambda W$. Note that $1_A = \sum_{g\in \Lambda} \eta_g$, by $(i)$ we obtain
$$
    f = \sum_{g\in \Lambda} f \eta_g = \sum_{g \le e} f \eta_g \in A_*{(e)}.
$$
It follows from $(ii)$ of Lemma \ref{centralizer} that $\sigma = w_1 f w_2 \in w_1 A_*{(e)} w_2 = \sum_{g\le e} w_1 A w_2 \eta_g = \sum_{g\le e} A\eta_g = A_*{(e)}$. This proves $(ii)$.

For $(iii)$, note that $\eta_e \in I_*(e)$ and $A\eta_e \subseteq I_*(e)$. If $\alpha \in A\eta_e$, by (\ref{ie}) we can write
\[
    \alpha = \sum_{\sigma \in WfW, ~f\le e} a_{\sigma} \sigma
\]
with $a_{\sigma} \in F$, and hence
$\alpha = \alpha\eta_e = \sum_{\sigma \in WeW} a_{\sigma} (\sigma \eta_e) $ by $(i)$. In other words, the set $\{ \sigma \eta_e \mid \sigma\in WeW \}$ spans $A\eta_e$. On the other hand, $\dim A = |R| = \sum_{e \in \Lambda} |WeW| \ge \sum_{e \in \Lambda} \dim A\eta_e = \dim A$. This indicates $\dim A\eta_e = |WeW|$, which proves $(iii)$.
$\hfill\Box$
\end{proof}

We describe structures of $A\eta_e$. Let $L=\varepsilon^{-1}(e)$ be the unique face of $P$ with $e = e_L$. Then for any $J \in \F(e)$ there is $w\in W$ such that $w(L) = J$ and $w^{-1} e_{L} w = e_J$. Hence
$$
    \mu_J = e_{L} w e_J
$$
is an element of $R$ and maps $L$ to $J$. Similarly $\mu_J^-= e_J w^{-1} e_{L} \in R$ and maps $J$ to $L$ and \begin{eqnarray}\label{kk}
    \mu_J^-\mu_J = e_J w^{-1} e_{L} w e_J = e_J.
\end{eqnarray}

We define a projection from $WeW$ to $W^*(e)$. If $\sigma\in WeW$ with $e\in \Lambda$ then from Lemma \ref{idempotent} there exit $w\in W$ and unique $I, J \in \mathcal F(P)$ such that
\begin{eqnarray}\label{sigmaJK}
    \sigma = e_I w e_J.
\end{eqnarray}
Consider $p(\sigma) = \mu_I\sigma\mu_J^-$. It follows easily that $p(\sigma)$ maps ${L}$ to ${L}$ and $p(\sigma) e = e p(\sigma)$, i.e., $p(\sigma) \in W(e)$. But then $p(\sigma) = e p(\sigma)$ forces  $p(\sigma)\in eW(e)$. However, $(iv)$ of Proposition \ref{kepprop} tells us that $eW(e) = W^*(e)$; we can now define
\begin{eqnarray}
    p(\sigma) =\mu_I\sigma\mu_J^-\in W^*(e).
\end{eqnarray}
By (\ref{kk}) and (\ref{sigmaJK}),
$
    \sigma = \mu_I^- \mu_I \sigma \mu_J^- \mu_J = \mu_I^-p(\sigma) \mu_J.
$
For any $\sigma\in R$, the $p(\sigma)$ is unique: If $I, J\in \F(P)$ and $\mu_I^-p(\sigma) \mu_J = \mu_I^-q(\sigma) \mu_J$ with $p(\sigma), q(\sigma)\in W^*(e)$ then $p(\sigma) = q(\sigma)$. If $\sigma\in W$ then $p(\sigma) = \sigma \in W$.

Let $B_e = F W^*(e)$ be the group algebra of $W^*(e)$ over $F$ where $e\in \Lambda $. Then each $B_e$ is a semisimple associative algebra \cite {NJ} and $\dim B_e = |W^*(e)|$, since the characteristic of $F$ is 0.  Let $d_e = |W| / |W(e)|$. Then
\begin{eqnarray}\label{ae}
    A_e = {\bf M}_{d_e}(B_e)
\end{eqnarray}
is the $F$-algebra of all matrices with entries in $B_e$ and rows and columns indexed by faces $I, J$ of $\F(e)$. It is easily seen that $\dim A_e = d_e^2 |W^*(e)|$.

If $e=0$, then $A_e = F$ and $\psi_e: A\eta_e = F e \rightarrow F$ given by $\psi_e(e) = 1$ is an algebraic isomorphism. If $e = 1$, then $A_{e} = B_{e} = FW$ and $\psi_{e}: A\eta_{e} \rightarrow FW$ defined by $\psi_{e}(\sigma \eta_{e}) = \sigma$ for $\sigma\in W$ is an isomorphism as algebras. More generally, we have the following result.

\begin{theorem}\label{isomorphism}
(i) The $A\eta_e$ is isomorphic to $A_e$ as an $F$-algebra.

(ii) $A \simeq \bigoplus_{e\in \Lambda } A_e$, and $A$ is a semisimple algebra.
\end{theorem}

\begin{proof}
Let $E_{IJ}$ be the natural basis of matrix units in ${\bf M}_{d_e}(B_e)$.
Define an $F$-linear map $\psi_e: A\eta_e \rightarrow A_e$ as follows. For any $\sigma \in WeW$, let $I = I(\sigma), J = J(\sigma) \in \F(P)$. Then $\sigma = w e_J = e_I w e_J$ for some $w\in W$. Define
\begin{eqnarray}\label{psidefinition}
    \psi_e( \sigma \eta_e ) = p(\sigma)E_{IJ}.
\end{eqnarray}
This definition is well defined by $(iii)$ of Lemma \ref{basis}. Write $\eta = \eta_e$ and $\psi = \psi_e$. To show that $\psi$ is a homomorphism of algebras, it is suffices to show that $\psi( (\sigma \eta) (\tau \eta) ) = \psi (\sigma \eta) \psi(\tau \eta)$ for $\sigma, \tau \in WeW$. Let $p=p(\sigma), q = p(\tau)$ and let $L = I(\tau), K=J(\tau)$. If $J=L$, then $\psi (\sigma \eta) \psi(\tau \eta) = (pE_{IJ}) (qE_{LK}) = pqE_{IK}$. If $J\ne L$, then $\psi (\sigma \eta) \psi(\tau \eta) = 0$.

On the other hand, if $J=L$, then $\sigma \tau = (\mu_I^-p\mu_J)(\mu_J^-q\mu_K) = \mu_I^- pq \mu_K$. This gives $\psi( \sigma \tau \eta ) = pqE_{IK}$. But $\psi( (\sigma \eta) (\tau \eta) ) = \psi( \sigma \tau \eta )$ by $(ii)$ of Lemma \ref{centralizer}. So $\psi( (\sigma \eta) (\tau \eta) ) = \psi (\sigma \eta) \psi(\tau \eta) = pqE_{IK}$. If $J\ne L$, then $\sigma\tau = (w e_J) (e_L w_1)$ for some $w_1 \in W$ since $L = I(\tau)$. By $(i)$ and $(ii)$ of Lemma \ref{conjugate} we have $e_{J\cap L} < e_J \in E(WeW)$, and so $\sigma \tau = w_1 e_{J\cap L} w_2 \in WfW$ for some $f < e$. Thus $(\sigma \tau) \eta = 0$ by Lemma \ref{basis} $(i)$. This indicates that $\psi ((\sigma \eta) (\tau \eta)) = 0 = \psi (\sigma \eta) \psi(\tau \eta)$ . Thus $\psi$ is a homomorphism of algebras.

To show that $\psi$ is one-to-one, let $a\in \ker \psi \subseteq A\eta_e$. Since $\sigma = \mu_I^-p\mu_J$, by Lemma \ref{basis} $(iii)$ we may write $a= \sum a_{IJ}(p) (\mu_I^-p\mu_J) \eta $ where $a_{IJ}(p) \in F$. The sum is over all faces $I, J \in \F(e)$ and all $p\in W^*(e)$. Then $0= \psi(a) = \sum_{IJ}\sum_{p} a_{IJ}(p) p E_{IJ}$. This indicates $\sum_{p} a_{IJ}(p)p=0$ for all $I, J$. Thus $a_{IJ}(p) = 0$ for all $I, J$ and $p$. So $\psi$ is one-to-one. By Lemma \ref{lemma de} and Lemma \ref{basis} $(iii)$, we obtain $\dim A\eta = |WeW| = d_e^2 |W^*(e)|=\dim A_e$. Thus $\psi$ is an isomorphism.

It follows from $(i)$ and (\ref {Adecompositin}) that $A \simeq \bigoplus_{e \in \Lambda} A_e$. Because the characteristic of $F$ is zero, the algebra $A$ is semisimple.         $\quad\Box$
\end{proof}

Next, we describe the representations of $A$ and $R$. Let $\psi_e$ be as in (\ref{psidefinition}). Similarly to \cite{LS2}, for $a\in A$ define $\beta_{ij}(a) \in B_e$ for $1 \le i, j \le d_e$ by
\begin{eqnarray}\label{psidefinition2}
    \psi_e(a\eta_e) = \sum_{i,j = 1}^{d_e} \beta_{ij}(a) E_{ij}.
\end{eqnarray}
Any representation $\rho$ of $B_e$ induces a representation $\rho^*$ of $A$ as follows
\begin{eqnarray}\label{rhostardefinition}
    \rho^*(a) = \sum_{i, j = 1}^{d_e} \rho\big(\beta_{i, j}(a)\big) E_{i, j}.
\end{eqnarray}
Then
\begin{eqnarray}
    \mbox{deg}\rho^* = d_e \mbox{deg} \rho.
\end{eqnarray}

Notice that the representations of $R$ are equivalent to the representations of $A = FR$. We have the following fact.

\begin{theorem}\label{theorem 3.1}
Suppose that $\widehat{B}_e$ is a full set of inequivalent irreducible representations of $B_e$, where $e\in \Lambda$. The set $\{\rho^* \mid \rho\in \widehat{B}_e \text{ for } e\in\Lambda\}$ is a full set of inequivalent irreducible representations of $R$.
\end{theorem}
\begin{proof}
The result follows from Lemma 2.22 of \cite{LS2} and Theorem \ref{isomorphism} above.       $\hfill\Box$
\end{proof}

\section{Character Formula}

In this section we give an analogue of the Munn-Solomon formula for calculating the character of the Renner monoid $R$ in terms of the characters of $W^*(e)$ for $e \in \Lambda$. To do this, we need to compute the value of the character of $\rho^*$ on any element of $R$ in terms of the faces in $\F(P)$.

\begin{proposition}\label{proposition 4.1} Let $e\in \Lambda$, and let $\rho$ be a representation of $W^*(e)$. Then for $\sigma\in R$,
\begin{eqnarray}\label{rhoStarProp}
    \rho^*(\sigma) = \sum_{ K \in \F(e), K \subseteq I( \sigma) } \rho(p(e_K \sigma))E_{I(e_K \sigma), J(e_K \sigma)}.
\end{eqnarray}
\end{proposition}

\begin{proof}

Suppose that $\rho$ is a representation of $W^*(e)$ with $e\ne 1$. If $\sigma\in WfW$ where $f<e$ or $f$ is not related to $e$, it follows from $(i)$ of Lemma \ref{basis} that $\sigma\eta_e = 0$. So
\begin{eqnarray}\label{fLesse}
\rho^*(\sigma) = 0      \mbox{\qquad   for } \sigma\in WfW \mbox{ where } f<e \mbox{ or } f \mbox{ is not related to } e.
\end{eqnarray}
On the other hand, there is no face $K$ with $K \in \F(e)$ and $K \subseteq I( \sigma)$. Thus the right hand side of (\ref{rhoStarProp}) is zero. This proves (\ref{rhoStarProp}) if $f<e$ or $f$ is not related to $e$.

If $\sigma\in WeW$, then $K=I(\sigma)$ is the only face of $P$ such that $K \in \F(e)$ and $K \subseteq I( \sigma)$. This tells us that $e_K \sigma = \sigma$ and so the right hand side of (\ref{rhoStarProp}) is $\rho(p(\sigma)) E_{IJ}$ with $I=I(\sigma)$ and $J=J(\sigma)$. On the other hand, $\psi_e(\sigma\eta_e) = p(\sigma)E_{IJ}$ by (\ref{psidefinition}). Therefore,
\begin{eqnarray}\label{sigmaInWeW}
\rho^*(\sigma) = \rho(p(\sigma))E_{IJ}     \quad\mbox{ if } \sigma\in WeW,
\end{eqnarray}
by (\ref {psidefinition}) and (\ref{rhostardefinition}). This proves (\ref{rhoStarProp}) if $\sigma\in WeW$.

Next, we suppose that $\sigma\in WfW$ where $f>e$. Denote by $Q = \bigcup_{g\in \Lambda, g < e} A\eta_g$. Let $\equiv$ be congruence mod $Q$. Then
\[
    \eta_e = \sum_{K \in \F(e)} \eta_K = \sum_{K \in \F(e)} \sum_{J\in \F(K)}(-1)^{\dim K - \dim J} e_J \equiv \sum_{K\in \F(e)} e_K.
\]
Thus
\[
    \sigma\eta_e = \eta_e \sigma\eta_e \equiv \sum_{K\in \F(e)} e_K\sigma \eta_e.
\]
By (\ref {psidefinition2}), (\ref{rhostardefinition}) and (\ref{fLesse}), we have
\[
    \rho^*(\sigma) = \sum_{K\in \F(e), K \subseteq I(\sigma)} \rho^*(e_K\sigma).
\]
Applying (\ref{sigmaInWeW}), we know that (\ref{rhoStarProp}) is true.             $\qquad\Box$
\end{proof}

\noindent {\bf Example 4.1} Let $e\in \Lambda$ be minimal and $\rho$ an irreducible representation of $W^*(e)$. Then the face $L=\varepsilon^{-1}(e)$ corresponding to $e$ is a vertex of $P$. So $\F(e) = \{wL\mid w\in W\}$ consists of some vertices of $P$. To determine $\rho^*$, use Proposition \ref{proposition 4.1}. The conditions $K\in \F(e)$ and $K \subseteq I( \sigma)$ on $K$ in $(4.1)$ are equivalent to that $K = \{ i \}$ and $i\in I(\sigma)$. Thus $\rho^*(\sigma) = \sum_{i \in I(\sigma)}\rho(p(e_{\{i\}}\sigma))E_{i, i\sigma}$.

\vspace{2mm}
\noindent {\bf Example 4.2} If $M$ is $\mathcal{J}$-irreducible and $e\in \Lambda$ is minimal, by $(iii)$ of Theorem \ref{recipe} we get $W^*(e) = 1$. Let $\rho$ be the only irreducible representation of $W^*(e)$ given by $\rho: W^*(e) \rightarrow F$ such that $\rho(W^*(e)) = 1 \in F$. Note that $p(e_K\sigma) \in W^*(e) = \{ 1 \}$; we have $\rho(p(e_K\sigma)) = 1$. It follows from Example 4.1 that
\begin{eqnarray} \label{matrixrep}
    \rho^*(\sigma) = \sum_{i \in I(\sigma)} E_{i, i\sigma}.
\end{eqnarray}
More generally, if $M$ is a reductive monoid with 0 and $e\in \Lambda$ minimal with $W^*(e) = 1$, then $\rho^*$ is given by (\ref{matrixrep}).

\vspace{2mm}
\noindent {\bf Example 4.3} Suppose that $e=1\in\Lambda$ and $\rho$ is a representation of $W$. Then $\F(e) = \{P\}$. If $\sigma\in W$, then the only face $K$ of $P$ with the conditions on $K$ in (\ref{rhoStarProp}) is $K=P$, in this case $\e_K \sigma = \sigma$. It follows from (\ref{rhoStarProp}) that $\rho^*(\sigma) = \rho(\sigma)$. If $\sigma\in R-W$, then $I(\sigma)$ is a proper subset of $V(P)$ and there is no face $K$ of $P$ such that $K \in \F(e)$ and $K \subseteq I( \sigma)$. So $\rho^*(\sigma) = 0$ by Proposition \ref{proposition 4.1} So $\rho^* = \rho \circ \pi$ where $\pi: FR \rightarrow FW$ is an algebraic homomorphism defined by $\pi(\sigma) = \sigma$ for $\sigma\in W$ and $\pi(\sigma) = 0$ for $\sigma\in R-W$.

\begin{theorem} \label{theorem 4.1}
Suppose that $e\in \Lambda$. Let $\chi$ be a character of $W^*{(e)}$. Let $\chi^*$ be the corresponding character of $R$. Then for $\sigma\in R$,
\begin{eqnarray}\label{chistar}
    \chi^*(\sigma) =  \sum_{ K \in \F(e), K \sigma = K }
    \chi ( \mu_K \sigma \mu^-_K ).
\end{eqnarray}
\end{theorem}

\begin{proof}
By Proposition \ref{proposition 4.1} it is easy to see that
\[
    \chi^*(\sigma) = \sum_{ K \in \F(e), K \subseteq I(\sigma) \atop {I(e_K \sigma) = J(e_K \sigma) }} \chi(p(e_K \sigma)).
\]
Notice that, for $K\in \F(P)$, the conditions $K \subseteq I(\sigma)$ and ${I(e_K \sigma) = J(e_K \sigma) }$ are the same as $K\sigma = K$. By the definition of $\mu$ we have $\mu_K e_K = \mu_K$ and $p(e_K \sigma) = \mu_K e_K \sigma \mu^-_K = \mu_K \sigma \mu^-_K$. Therefore, (\ref{chistar}) is true.    $\qquad\Box$
\end{proof}

\noindent {\bf Example 4.4} Let $\mu$ be the first fundamental dominant weight of type $C_3$ and $M=\overline{F^*Sp_6}$, the Zariski closure of $F^*Sp_6$ in $M_6(F)$. Let $\{\varepsilon_1, \varepsilon_2, \varepsilon_3\}$ be a standard orthogonal basis of a 3-dimensional Euclidean space. Then the root system of $C_3$ is
$$
\Phi=\{\pm \varepsilon_i\pm \varepsilon_j,~1\leq i<j\leq 3;~\pm 2\varepsilon_i,~1\leq i\leq 3\}
$$
Then $\mu=\varepsilon_1$ and
$$
    W(\mu) = \{~\pm \varepsilon_1,~ \pm \varepsilon_2, ~\pm \varepsilon_3\}.
$$
The polytope $P$ is the {\it octahedron}, a three-dimensional crosspolytope labeled as follows. It is simplicial.

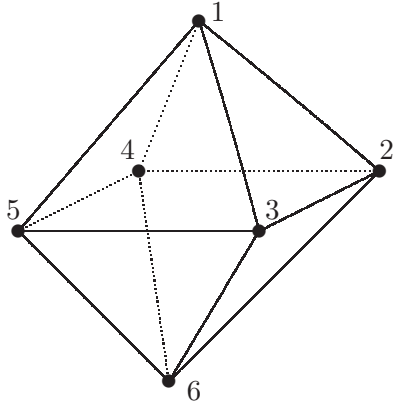
\begin{center}
\setlength{\unitlength}{0.8mm}
\begin{picture}(60,70)(0,-30)
\put(0,0){\circle*{2}}
\put(20,10){\circle*{2}}
\put(40,0){\circle*{2}}
\put(60,10){\circle*{2}}
\put(30,35){\circle*{2}}
\put(25,-25){\circle*{2}}
\qbezier(0,0)(15,17.5)(30,35)
\qbezier[30](20,10)(25,22.5)(30,35)
\qbezier(60,10)(45.5,22.5)(30,35)
\qbezier(40,0)(35.5,17.5)(30,35)
\qbezier[30](20,10)(25,22.5)(30,35)
\qbezier[50](20,10)(40,10)(60,10)
\qbezier(40,0)(50,5)(60,10)
\qbezier(0,0)(20,0)(40,0)
\qbezier[30](0,0)(10,5)(20,10)
\qbezier(25,-25)(12.5,-12.5)(0,0)
\qbezier[50](25,-25)(22.5,-7.5)(20,10)
\qbezier(25,-25)(32.5,-12.5)(40,0)
\qbezier(25,-25)(42.5,-7.5)(60,10)

\put (-2,  2) {$5$}
\put(17,12){$4$}
\put(41,2){$3$}
\put(60,12){$2$}
\put(32,35){$1$}
\put(28,-28){$6$}

\end{picture}
\end{center}
It follows that
\begin{eqnarray*}
E(\overline T)  &\simeq& \F(P)\\
                &=&\big\{ \emptyset, \{1\}, \{2\}, \{3\}, \{4\}, \{5\}, \{6\}, \{1, 2\}, \{1, 3\}, \{1, 4\}, \{1, 5\}, \{6, 2\}, \{6, 3\},\\
                &&~\{6, 4\}, \{6, 5\}, \{2, 3\}, \{2, 4\}, \{3, 5\}, \{4, 5\}, \{1, 2, 3\}, \{1, 2, 4\}, \{1, 3, 5\},\\
                &&~\{1, 4, 5\}, \{6, 2, 3\}, \{6, 2, 4\}, \{6, 3, 5\}, \{6, 4, 5\}, \{1, 2, 3, 4, 5, 6\} \big\}.
\end{eqnarray*}
Let $e\in E(\overline T)$ correspond to the $2$-dimensional face $\{1, 2, 3\}$. Then $\F(e)$ consists of all $2$-dimensional faces of $P$, i.e.,
\[
    \F(e) = \big\{\{1, 2, 3\}, \{1, 2, 4\}, \{1, 3, 5\}, \{1, 4, 5\}, \{6, 2, 3\}, \{6, 2, 4\}, \{6, 3, 5\}, \{6, 4, 5\},  \big\}
\]
Suppose that $I(\sigma) = \{1, 2, 3, 4, 5, 6\}$ and that $\sigma: 1\mapsto 2\mapsto 3\mapsto 1$ and $4\mapsto 6\mapsto 5\mapsto 4$. Then $K = \{1, 2, 3\} $ and $K = \{4, 5, 6\}$ are the only two faces in $\F(e)$ such that $K\sigma = K$. For $K = \{1, 2, 3\}$, choose $\mu_K$ to be the identity map on K. Thus $\mu_K \sigma \mu^-_K = (123)$ in the usual cycle notation for permutations. For $K = \{4, 5, 6\}$, choose $\mu_K: \{1, 2, 3\} \rightarrow \{4, 5, 6\}$ with
$\mu_K: 1\mapsto 4, 2\mapsto 5, 3\mapsto 6$. The domain of $\mu_K \sigma \mu^-_K$ is $\{1, 2, 3\}$ and $\mu_K \sigma \mu^-_K = (123)$. It follows from Theorem \ref{theorem 4.1} that $\chi^*(\sigma) = 2\chi((123))$.

\vspace{2mm}
\noindent{\bf Acknowledgment} We would like to thank Mohan Putcha and Lex Renner for their help when we were stuck. We are grateful to Benjamin Steinberg for bringing our attention to his papers \cite{S1, S2}. We thank Mahir Can for email communications that are useful in the development of this paper. We would also like to thank Reginald Koo for his help with drawing the polytope in Example 4.4 using Latex.

\pagebreak
\vspace{1cm}
\noindent Zhuo Li \\
Department of Mathematics \\
Xiangtan University\\
Xiangtan, Hunan 411105, P. R. China\\
\noindent Email: zli@xtu.edu.cn\\

\noindent Zhenheng Li\\
Department of Mathematical Sciences \\
University of South Carolina Aiken\\
Aiken, SC 29801, USA\\
\noindent Email: zhenhengl@usca.edu\\

\noindent You'an Cao \\
Department of Mathematics \\
Xiangtan University\\
Xiangtan, Hunan 411105, P. R. China\\
\noindent Email: cya@xtu.edu.cn\\


\begin{thebibliography}{1}
\bibitem{CP} A. Clifford and G. B. Preston, The Algebraic Theory of Semigroups, Mathematical
Surveys No. 7, AMS, Providence, RI, Vol. 1, 1961.

\bibitem{BG} B. Grunbaum, Convex Ploytopes, Springer, 2003.

\bibitem{NJ} N. Jacobson, Basic Algebra II, W. H. Freeman and Company, San Francisco, 1980.

\bibitem{ZHEN} Z. Li, Z. Li and Y. Cao, Representations of the symplectic rook monoid, International Journal of Algebra and Computation, 18(5) 2008, 837-852.

\bibitem{LLC} Z. Li, Z. Li and Y. Cao, Orders of the Renner monoids, J. Algebra, 301(1) 2006, 344-359.

\bibitem{LR} Z. Li and L. Renner, The Renner monoids and cell decompositions of the symplectic algebraic monoids,
International Journal of Algebra and Computation, Vol. 13, No. 2 (2003) 111-132.

\bibitem{M0} W. Munn, On semigroup algebras, Proc. Cambridge Philos. Soc. 51, (1955). 1-15.

\bibitem{M1} W. Munn, Matrix representations of semigroups, Proc. Camb. Phil. Soc., 53 (1957) 5-12.

\bibitem{M2} W. Munn, The characters of the symmetric inverse semigroup, Proc. Camb. Phil. Soc., 53 (1957) 13-18.

\bibitem{O} J. Okn\'{\i}nski, Semigroup Algebras, Monographs and Textbooks in Pure and Applied
Mathematics, vol. 138, Marcel Dekker, Inc., New York, 1991.

\bibitem{PO} I. Ponizovskii, On matrix representations of associative systems, Mat. Sb. N.S. 38,
(1956), 241-260.

\bibitem{P} M. Putcha, On linear algebraic semigroups III, Internat. J. Math. \& Math. Sci 4 (1981) 667-690.

\bibitem{P1} M. Putcha, Linear Algebraic Monoids, London Math. Soc. Lecture Note Series 133, Cambridge Unive. Press, 1988.

\bibitem{PU2} M. Putcha, Sandwich matrices, Solomon algeras, and Kazhdan-Lusztig polynomials, Trans. Amer. Math. Soc., 340(1) (1993), 415-428.
\bibitem{PU3} M. Putcha, Complex representations of finte monoids, Proc. London Math. Soc., 73 (1996), 623-641.

\bibitem{PU4} M. Putcha, Monoid Hecke algebras, Trans. Amer. Math. Soc., 349 (1997), 3517-3534.
\bibitem{PU5} M. Putcha, Complex representations of finte monoids II. Highest weight categories and quivers, Journal of Algebra, 205 (1998), 53-76.

\bibitem{PU7} M. Putcha, M\"{o}bius functions on cross section lattices, Journal of Combinatorial Theory Series A, Vol 106(2004), 287-297.

\bibitem{P2} M. Putcha, Bruhat-Chevalley Order in Algebraic Monoids, J. of Algebraic Combinatorics, 20(2004), 33-53.

\bibitem{P3} M. Putcha, Nilpotent variety of a reductive monoid, J. of Algebraic Combinatorics, to apear.

\bibitem{P4} M. Putcha, M\"{o}bius functions on cross section lattices, Journal of Combinatorial Theory Series A, Vol 106(2004), 287-297.

\bibitem{PR1} M. Putcha and L. Renner, The System of Idempotents and Lattice of $J$-classes of Reductive Algebraic Monoids, J. Algebra 226(1988), 385-399.

\bibitem{PR2} M. Putcha and L. Renner, The canonical comactification of a finite group of Lie type, Trans. Amer. Math. Soc., 337 (1993), 305-319.

\bibitem{R1} L. Renner, Analogue of the Bruhat Decomposition for Algebraic Monoids, J. of Algebra 101 (1986) 303-338.

\bibitem{R2} L. Renner, Modular representations of fintie monoids of Lie type, Semigroups, Formal Language and Groups, 1995, 381-390.

\bibitem{R3} L. Renner, Modular representations of fintie monoids of Lie type, Journal of Pure and Applied Algebras, 138(1999), 279-296.

\bibitem{R4} L. Renner, Linear Algebraic Monoids, Series: Encyclopedia of Mathematical Sciences, Springer-Verlag, Vol 134, 2005.

\bibitem{LS0} L. Solomon, The Burnside algebra of a finite group, J. Combin. Theory 2 (1967), 603-615.

\bibitem{LS1} L. Solomon, The Bruhat decomposition, Tits system and Iwahori ring for the monoid
of matrices over a finite field. Geom. Dedicata 36 (1990) 15-49.

\bibitem{LS2} L. Solomon, Representation of the rook monoid, J. Algebra 256(2002), 309-342.

\bibitem{LS3} L. Solomon, An introduction to reductive monoids, Semigroups, Formal Languages and Groups, J. Fountain, ED., Kluwer Academic Publishers, 1995, 295-352.

\bibitem{RS} R. Stanley, Enumerative Combinatorics. Vol. 1, Cambridge Studies in Advanced Mathematics, vol.49, Cambridge University Press, 1997.

\bibitem{S1} B. Steinberg, M\"{o}bius functions and semigroup representation theory, Journal of Combinatorial Theory, Series A 113 (2006) 866-881.

\bibitem{S2} B. Steinberg, M\"{o}bius Functions and Semigroup Representation Theory II: Character formulas and multiplicities, Advances in Mathematics, Adv. in Math., 217 (2008), 1521-1557.

\bibitem{GZ} G. Ziegler, Lectures on polytopes, Graduate Texts in Mathematics, Vol. 152, Springer-Verlag, New York, 1995.

\end{thebibliography}
\end{document}